\documentclass[11pt]{article}
\usepackage[a4paper, margin=1in]{geometry}
\usepackage[utf8]{inputenc}
\usepackage{amsmath}
\usepackage{amssymb}
\usepackage{amsthm}
\usepackage{mathtools}
\usepackage{graphicx, float}
\usepackage{xcolor}
\usepackage{epstopdf}
\usepackage{enumitem}
\usepackage{accents}
\usepackage{bm,bbm}

\newtheorem{theorem}{Theorem}[section]
\newtheorem{lemma}[theorem]{Lemma}
\newtheorem{corollary}[theorem]{Corollary}
\newtheorem{proposition}[theorem]{Proposition}
\newtheorem{claim}[theorem]{Claim}

\newtheorem{obs}[theorem]{Observation}

\newtheorem{constr}[theorem]{Construction}
\newcommand{\ceil}[1]{\left\lceil{#1}\right\rceil}
\newcommand{\floor}[1]{\left\lfloor{#1}\right\rfloor}

\newcommand{\ctmod}[2]{\floor{\frac{#1}{#1-#2}}(#1-#2)}
\newcommand{\floorkl}{\floor{\frac{k}{k-\ell}}(k-\ell)}
\newcommand{\ceilkl}{\ceil{\frac{k}{k-\ell}}}
\newcommand{\dconnect}{1-\frac{1}{\ceilkl}}
\newcommand{\dcover}{1-\frac{1}{\floorkl}}

\begin{document}
\title{Positive codegree thresholds for Hamilton cycles in hypergraphs}
\author{Richard Mycroft\thanks{School of Mathematics, University of Birmingham, UK. {\tt r.mycroft@bham.ac.uk}. RM is grateful for financial support from EPSRC Standard Grant EP/R034389/1.} \and
 Camila Z\'arate-Guer\'en\thanks{School of Mathematics, University of Birmingham, UK. {\tt ciz230@student.bham.ac.uk}.}}
\date{}
\maketitle

\begin{abstract}
    For each $k \geq 3$ and $1 \leq \ell \leq k-1$ we give an asymptotically best possible minimum positive codegree condition for the existence of a Hamilton $\ell$-cycle in a $k$-uniform hypergraph. This result exhibits an interesting duality with its analogue under a minimum codegree condition. The special case $\ell = k-1$ of our result establishes an asymptotic version of a recent conjecture of Illingworth, Lang, M\"uyesser, Parczyk and Sgueglia on tight Hamilton cycles in hypergraphs.
\end{abstract}

\section{Introduction}

Over recent decades, the problem of finding hypergraph analogues of Dirac's theorem~\cite{dirac52} -- that every graph $G$ on $n \geq 3$ vertices with minimum degree $\delta(G) \geq n/2$ admits a Hamilton cycle -- has been the subject of extensive research in extremal graph theory. This is because the question is a natural fundamental problem concerning connected spanning structures which has spurred the development of versatile new techniques that have proved useful much more widely, but also because of the range of possible generalisations: as well as distinct uniformities $k$ for the host hypergraph, there are many natural ways to generalise the concept of a cycle and of degree. For some of these generalisations the problem has been completely resolved, whilst for others we have only conjectures to which the solution remains far out of reach (a more extensive summary of previous work is included in Section~\ref{sec:pastwork}).

Minimum codegree conditions are the strongest commonly-used form of degree condition for a $k$-graph (i.e. $k$-uniform hypergraph) $H$: saying that $H$ has minimum codegree at least $\delta$ means that every set of $k-1$ vertices of $H$ is contained in at least $\delta$ edges of $H$. This condition proves very powerful for constructive arguments, and so the problem of generalising Dirac's theorem to the hypergraph setting under minimum codegree conditions has been essentially resolved. On the other hand, minimum codegree conditions are too strong for many natural settings. For example, multipartite $k$-graphs have received much research attention in their own right, but by definition have minimum codegree zero. More generally, say that a set $S$ in a $k$-graph $H$ is a strong independent set if no edge of $H$ has more than one vertex in $S$ (for example, this is true of each vertex class in a $k$-partite $k$-graph $H$). The existence of a strong independent set~$S$ in~$H$  which contains at least two vertices forces the minimum codegree of $H$ to be zero.

An ideal notion of minimum degree for hypergraphs would retain the constructive power of the minimum codegree condition but also be sufficiently versatile to be applicable in $k$-graphs with non-trivial strong independent sets. This motivates the consideration of minimum \emph{positive} codegree conditions, where the codegree condition applies only to sets of $k-1$ vertices which are contained in at least one edge of $H$. To be more precise, let $H$ be a $k$-graph. We write $\deg(S)$ for the degree of a set $S \subseteq V(H)$, which is the number of edges containing $S$ as a subset. The minimum positive codegree of $H$ a $k$-graph $H$, denoted $\delta^+(H)$, is the minimum of $\deg(S)$ over all sets $S$ of $k-1$ vertices of $H$ which satisfy $\deg(S) \geq 1$. By contrast, the minimum codegree of $H$, denoted $\delta(H)$, is the minimum of $\deg(S)$ over all sets $S$ of $k-1$ vertices of $H$. Note that a minimum positive codegree condition on $H$ does not prohibit the existence of isolated vertices on $H$, so when we wish to find spanning structures in $H$ we also require the (very weak) condition that $H$ has no isolated vertices. 

The main result of this manuscript is to provide an analogue of Dirac's theorem for $k$-uniform hypergraphs under minimum positive codegree conditions. The structures we obtain are the most-studied generalisations of cycles in graphs to the $k$-graph setting. Specifically, for $1 \leq \ell \leq k-1$ a $k$-uniform $\ell$-cycle is a $k$-graph $C$ with no isolated vertices whose vertices are cyclically ordered so that every edge consists of $k$ consecutive vertices, and adjacent edges intersect in precisely $\ell$ vertices. Note in particular that the number of vertices in a $k$-uniform $\ell$-cycle must be a multiple of $k-\ell$, so the condition that $k-\ell$ divides $n$ is necessary for the existence of a Hamilton $\ell$-cycle in a $k$-graph on $n$ vertices. Our main result gives for every $k \geq 3$ and $1 \leq \ell \leq k-1$ a minimum positive codegree condition which is sufficient to ensure the existence of a Hamilton $\ell$-cycle in a $k$-graph and which is best possible up to the $\alpha n$ error term (see Section~\ref{sec:optimality} for further details on the optimality).

\begin{theorem} \label{main}
For all $k \geq 3$, $1\leq \ell \leq k- 1$ and $\alpha > 0$ there exists $n_0$ such that the following holds for every $n \geq n_0$ which is divisible by $k-\ell$. If $H$ is a $k$-graph on $n$ vertices with 
$$\delta^+(H) \geq \left(\dcover \right)n + \alpha n,$$
and $H$ has no isolated vertices,
then $H$ contains a Hamilton $\ell$-cycle.
\end{theorem}

A $k$-uniform $(k-1)$-cycle is also referred to as a tight cycle. For this case Illingworth, Lang, M\"uyesser, Parczyk and Sgueglia~\cite{ILMPS} recently conjectured that every $k$-graph on $n$ vertices with $\delta^+(H) \geq \frac{k-1}{k} n$ and no isolated vertices contains a tight Hamilton cycle. The case $\ell = k-1$ of Theorem~\ref{main} confirms that this conjecture holds asymptotically: every $k$-graph on $n$ vertices with $\delta^+(H) \geq \frac{k-1}{k} n + o(n)$ contains a tight Hamilton cycle. While we were finalising this manuscript, Letzter and Ranganathan announced a proof of the exact version of this conjecture; we understand that their work only considers tight Hamilton cycles and not Hamilton $\ell$-cycles for any $\ell \leq k-2$.

\subsection{Connection to previous work} \label{sec:pastwork}

\subsubsection{Minimum codegree conditions for Hamilton cycles}
The question of finding best possible minimum codegree conditions which ensure the existence of a Hamilton cycle in a $k$-graph on $n$ vertices -- extending Dirac's theorem to the hypergraph setting -- was a key goal in extremal graph theory for many years. We summarise the results briefly here, and recommend the surveys of K\"uhn and Osthus~\cite{KOSurvey}, R\"odl and Ruci\'nski~\cite{RRSurvey} and Zhao~\cite{ZhaoSurvey} for a more detailed exposition. The first non-trivial bounds on the minimum codegree which guarantees a tight Hamilton cycle in a $k$-graph were provided by Katona and Kierstead~\cite{hh-katona99}, and the best possible condition was then determined asymptotically first for $k=3$ then for every $k \geq 3$ by  R\"odl, Ruci\'nski and Szemer\'edi~\cite{hh-rodl06, hh-rodl08}. These results in fact establish an asymptotically best possible minimum codegree condition for a Hamilton $\ell$-cycle in a $k$-graph for  each $\ell \in [k-1]$ for which $\ell$ divides $k$. Finally, for those $\ell \in [k-1]$ for which $k-\ell$ does not divide $k$, the best possible minimum codegree condition for a Hamilton $\ell$-cycle in a $k$-graph was determined by a series of works by K\"uhn and Osthus~\cite{pp-kuhn06-cherry}, Keevash, K\"uhn, Mycroft and Osthus~\cite{hh-keevash11}, H\`an and Schacht~\cite{hh-han10} and K\"uhn, Mycroft and Osthus~\cite{KMO14}. Collectively these results give the following theorem, which establishes an asymptotically best possible minimum codegree condition for a Hamilton $\ell$-cycle in a $k$-graph for every $\ell$ and $k$.

\begin{theorem}[\cite{hh-han10, hh-keevash11, KMO14, pp-kuhn06-cherry, hh-rodl06, hh-rodl08}] \label{codeg}
  For all $k \geq 3$, $1 \leq \ell < k$ and $\alpha > 0$, there exists~$n_0$ such that if $n \geq
  n_0$ is divisible by $k-\ell$ and $H$ is a $k$-graph on $n$ vertices with $$\delta(H) \geq 
  \begin{cases}
    \left( \frac{1}{2} + \alpha \right) n& \mbox{ if $k-\ell $ divides $k$,} \\
    \left(\frac{1}{\lceil 
    \frac{k}{k-\ell} \rceil(k-\ell)}+\alpha\right) n & \mbox{otherwise,}
  \end{cases} 
  $$ 
  then $H$ contains a Hamilton $\ell$-cycle. 
\end{theorem}

Ideally we would be able to identify the best possible minimum codegree condition for a Hamilton $\ell$-cycle in a $k$-graph precisely -- at least for large $n$ -- rather than asymptotically. This has been done for the case $k=3, \ell=2$ by R\"odl, Ruci\'nski and Szemer\'edi~\cite{RoRuSz09}, for the case $k=3, \ell=1$ by Czygrinow and Molla~\cite{CM}, for the cases $k \geq 3$ and $\ell < k/2$ by Han and Zhao~\cite{HZ}, for the case $k=4$ and $\ell=2$ by Garbe and Mycroft~\cite{GM}, and for the case $k \geq 6$ and $\ell = k/2$ by H\`an, Han and Zhao~\cite{HHZ}. To our knowledge the precise best possible condition remains an open question for all other cases of $k$ and $\ell$.

\subsubsection{Minimum positive codegree conditions}
The study of minimum positive codegree conditions was proposed by Balogh, Lemons and Palmer~\cite{BLP}, who used it to give a generalisation of the Erd\H{o}s-Ko-Rado theorem on intersecting families. The notion quickly saw further attention in hypergraph Tur\'an theory. Halfpap, Lemons and Palmer~\cite{HLP} established for several fixed $3$-graphs $F$ the asymptotic minimum positive codegree threshold for the existence of a copy of $F$ in a $3$-graph $H$ on $n$ vertices, providing an interesting variation on the well-studied analogous question for minimum codegree. They also studied the ``jumps" in this parameter, a question which was recently further addressed by Balogh, Halfpap, Lidick\'y, and Palmer~\cite{BHLP}.

We turn now to the question of finding spanning structures in $k$-graphs satisfying minimum positive codegree conditions. Let $H$ be a $k$-graph on $n$ vertices with no isolated vertices. Halpap and Magnan~\cite{HM} showed that for $k=3$, if $\delta^+(H) \geq 2n/3-1$ and $3$ divides $n$ then~ $H$ contains a perfect matching, and moreover that this minimum positive codegree condition is best possible in an exact sense. They also gave an asymptotically best possible minimum positive codegree condition for a perfect matching in $H$ for each $k \geq 4$. In the same paper they established a precise best possible minimum positive codegree condition on $H$ for the existence of a Hamilton Berge cycle (a weaker notion of $k$-graph cycle). Finally, they showed that if $k=3$ and $\delta^+(H) \geq n/2 + o(n)$ then $H$ contains a Hamilton 1-cycle (also known as a loose Hamilton cycle), assuming the necessary condition that $n$ is even, and gave a construction to demonstrate that this condition is best possible up to the $o(n)$ error term. Note that this result coincides with the case $k=3, \ell=1$ of Theorem~\ref{main}.

Illingworth, Lang, M\"uyesser, Parczyk and Sgueglia~\cite{ILMPS} studied the minimum positive codegree needed to ensure the existence of a spanning sphere in a $k$-graph $H$ on $n$ vertices, and showed that the condition $\delta^+(H) \geq n/2 +o(n)$ suffices. This confirmed the asymptotic form of a conjecture by Georgakopoulos, Haslegrave, Montgomery and Narayanan~\cite{GHMN} that the condition $\delta^+(H) \geq n/2$ should suffice for this. The latter authors also observed that this conjecture, if true, would be best possible. 

\subsection{Discussion}

Our main result, Theorem~\ref{main}, establishes an asymptotically best possible minimum positive codegree condition for the existence of a Hamilton $\ell$-cycle in a hypergraph. This can be seen as a direct parallel of Theorem~\ref{codeg} -- the product of a longterm sequence of work by many researchers -- which establishes the best possible minimum codegree condition for this property. 

We perceive an interesting duality between the two results. Specifically, suppose that for a given $k$-graph $F$ on $n$ vertices we want to construct a $k$-graph $H$ on $n$ vertices which does not contain $F$ as a spanning subgraph. One way to do this is to ensure that $H$ contains an independent set $I$ such that $V(H) \setminus I$ is smaller than any vertex cover of $F$. Indeed, if there is a copy of $F$ in $H$ then the vertices of $F$ not in $I$ must form a vertex cover of $F$. Another possibility is to ensure that $H$ contains a strong independent set $S$ whose size is larger than the largest strong independent set in $F$. Indeed, if there is a copy of $F$ in $H$ then the vertices of $F$ in $S$ must form a strong independent set in $F$. Observe that the sizes of $I$ and $S$ are determined by dual parameters: the sizes of the minimum vertex cover and maximum strong independent set, respectively.

The link to Theorems~\ref{main} and~\ref{codeg} is that a non-trivial minimum codegree condition on $H$ prohibits the existence of a strong independent set in $H$ containing more than one vertex. On the other hand, a minimum vertex cover of a $k$-graph $\ell$-cycle on $n$ vertices has size around $\frac{n}{\lceil \frac{k}{k-\ell} \rceil(k-\ell)}$. This gives rise to the minimum codegree condition of Theorem~\ref{codeg} for the case where $k-\ell$ does not divide $k$ (and even in the case where $k-\ell$ does divide $k$, this is a necessary and sufficient condition to give a path tiling lemma analogous to our Lemma~\ref{coverlemma}): this is the minimum codegree on $H$ needed to ensure that every independent set in $H$ excludes as least $\frac{n}{\lceil \frac{k}{k-\ell} \rceil(k-\ell)}$ vertices. 
By contrast, minimum positive codegree conditions do not prevent the existence of strong independent sets. Instead, the condition of Theorem~\ref{main} is asymptotically best possible to prevent the existence of a strong independent set of size at least $\frac{n}{\lfloor \frac{k}{k-\ell} \rfloor(k-\ell)}$, which is the approximate size of a largest strong independent set in an $\ell$-path or $\ell$-cycle $k$-graph (see Proposition~\ref{indepsetsinpath}).

As described in the introduction, we see minimum positive codegree conditions as being a very natural and logical object of study. We hope that future work will establish the minimum positive codegree thresholds for other spanning structures in $k$-graphs, such as perfect $F$-tilings for arbitrary $k$-graphs $F$. We note that determining best possible minimum codegree conditions for $F$-tilings in $k$-graphs remains a fiendishly difficult problem, with almost all cases still unsolved for $k\geq 3$. Our hope is that identifying best possible minimum positive codegree conditions for these structures will prove to be a more tractable problem that will also help to advance our understanding of the minimum codegree thresholds. We expect that these questions will see similar duality between the minimum codegree thresholds and minimum positive codegree thresholds as in the situation described above for Hamilton $\ell$-cycles.

It would be nice to improve Theorem~\ref{main} to an exact best possible minimum codegree condition, at least for large $n$. However, the analysis of cases in which $H$ is close to the extremal construction threatens to be technically challenging; in particular we note that even after many years the analogous problem of improving Theorem~\ref{codeg} to a best possible exact statement remains open for all $\ell > k/2$ except for the case $\ell=2, k=3$.

Finally, we highlight the use of a \emph{weighted} perfect fractional matching in the proof of Lemma~\ref{coverlemma}. In our opinion this is a versatile tool which is likely to be useful to give simpler arguments for many related problems.

\subsection{Notation and preliminaries}

We use the following, standard, notation in this manuscript. For each $n \in \mathbb{N}$ we write $[n] = \{1, 2, \dots, n\}$, and for a set $X$ we write $\binom{X}{n}$ to denote the set of all subsets of $X$ of size $n$. 

Let $H$ be a $k$-graph, and fix sets $S, W \subseteq V(H)$. We have already defined the degree $\deg(S)$ of $S$ to be the number of edges of $H$ which contain $S$ as a subset. In the case $|S| = k-1$ an equivalent statement is that $\deg(S)$ is the number of vertices $x \in V(H)$ for which $S \cup \{x\}$ is an edge of $H$. In this case we also define $\deg(S, W)$ to be the number of vertices $x \in W$ for which $S \cup \{x\}$ is an edge of $H$, and we write $N(S)$ for the neighbourhood of $S$, that is, the set of vertices $x \in V(H)$ for which $S \cup \{x\}$ is an edge of $H$ (so $|N(S)| = \deg(S)$). Wherever necessary we avoid ambiguity by writing $N_H(S)$, $\deg_H(S)$ or $\deg_H(S, W)$ in place of $N(S)$, $\deg(S)$ or $\deg(S, W)$ respectively.

A $k$-graph $H$ is $r$-partite if it is possible to partition $V(H)$ into $r$ sets, called vertex classes, so that every edge of $H$ has at most one vertex in each vertex class. If $F$ and $H$ are $r$-partite $k$-graphs with vertex classes $U_1, \dots, U_r$ and $V_1, \dots V_r$ respectively (including a specified order on the vertex classes of each) then a \emph{partition-respecting copy} of $F$ in $H$ is a copy of $F$ in $H$ in which for each $i \in [r]$ the vertices of $U_i$ are embedded in $V_i$.

We write $x \ll y$ to mean that for all $y > 0$ there exists $x_0 > 0$ such that for every $x > 0 $ with $x < x_0$ the subsequent statements hold, and define similar statements with more variables similarly. Also we write $a=b\pm c$ to mean that $b-c \leq a \leq b+c$.

We use the following concentration bound on binomial distributions. 

\begin{theorem}\cite[Corollary 2.3]{JLR00:randomgraphs} \label{chernoff}
  If~$0<a<3/2$ and~$X$ is a binomial random variable then~$\mathbb{P}\bigl(\,|X-\mathbb{E} X|\geq a\mathbb{E} X\,\bigr)\leq 2\exp(-a^2 \mathbb{E} X/3)$.
\end{theorem}

The next proposition is very useful for constructing edges one vertex at a time using a minimum positive codegree condition.

\begin{proposition} \label{choosevs}
Let $H$ be a $k$-graph. The following properties hold for every set $S \subseteq V(H)$ with $|S| \leq k-1$ and $\deg(S) \geq 1$. 
\begin{enumerate}[label=(\alph*), noitemsep]
    \item There are at least $\delta^+(H)$ vertices $x \in V(H)$ for which $\deg(S \cup \{x\}) \geq 1$.
    \item For every set $U \subseteq V(H)$ with $|U| > |V(H)| - \delta^+(H)$ we have $\deg_{H[U]}(S) \geq 1$.
    \item $\deg(S) \geq \binom{\delta^+(H)}{k-|S|}$.
\end{enumerate}
\end{proposition}

\begin{proof}
For (a) note that since $\deg(S) \geq 1$ there exists an edge $e \in E(H)$ with $S \subseteq e$. Choose any $y \in e \setminus S$, and let $T := e \setminus \{y\}$, so $T \in \binom{V(H)}{k-1}$, $\deg(T) \geq 1$ and $S \subseteq T$. It follows from the definition of $\delta^+(H)$ that there are at least $\delta^+(H)$ vertices $x \in V(H)$ for which $T \cup \{x\} \in E(H)$. For each such $x$ we have $S \cup \{x\} \subseteq T \cup \{x\}$ and therefore $\deg(S \cup \{x\}) \geq 1$.

For (b) let $r = k-|S|$, and choose vertices $x_1, x_2, \dots, x_r \in U$ in turn so that for each $i \in [r]$ we have $\deg(S \cup \{x_1, \dots, x_i\}) \geq 1$. This is possible since by (a) at least $\delta^+(H)$ vertices $x_i$ have this property, and fewer than $\delta^+(H)$ are not in $U$.

For (c) note that by (a) there are at least $\delta^+(H)^{r}$ ordered sequences $(x_1, x_2, \dots, x_r)$ for which $S \cup \{x_1, x_2, \dots, x_r\} \in E(H)$. Since there are only $r!$ possible orders for a set of $r$ vertices, the desired bound follows. 
\end{proof}

We also use the following simple observation to compare the minimum positive codegree conditions of different results we use in our proof.

\begin{obs} \label{observation} For every $k \geq 3$ and $1 \leq \ell \leq k-1$ we have $\ctmod{k}{\ell} \geq \ceilkl.$
\end{obs}

To see this, observe that if $\ell = k-1$ then both sides are equal to $k$, whilst otherwise we have $k-\ell \geq 2$ and $\floor{\frac{k}{k-\ell}} \geq 1$ so $\ctmod{k}{\ell} \geq 2 \floor{\frac{k}{k-\ell}} \geq \floor{\frac{k}{k-\ell}} + 1 \geq \ceilkl$. 
\medskip

For $1 \leq \ell \leq k-1$ a $k$-uniform $\ell$-path is a $k$-graph $P$ with no isolated vertices whose vertices are ordered linearly so that every edge consists of $k$ consecutive vertices and adjacent edges intersect in precisely $\ell$ vertices (so the difference between an $\ell$-path and an $\ell$-cycle is whether the order of the vertices is linear or cyclical). Let $P$ be an $\ell$-path with vertices $v_1, \dots, v_n$, numbered in the order they appear in $P$. Observe that the edges of $P$ must then be the sets $v_{(i-1)(k-\ell) + 1}, \dots, v_{(i-1)(k-\ell)+k}$ for each $i \geq 1$ with $(i-1)(k-\ell)+k \leq n$. Since $v_n$ is not an isolated vertex it follows that $n = (m-1)(k-\ell)+k = m(k-\ell) +\ell$ for some $m \geq 1$, and that~$P$ has $m$ edges; we call $m$ the \emph{length} of $P$. Similarly, if $P$ is an $\ell$-cycle with vertices $v_1, \dots, v_n$, numbered in the order they appear in $P$, and $\{v_1, \dots, v_k\}$ is an edge of $C$, then the edges of~$C$ are precisely the sets $v_{(i-1)(k-\ell) + 1}, \dots, v_{(i-1)(k-\ell)+k}$ for each $i  \in [n/(k-\ell)]$, with subscripts taken modulo $n$. So we must have that $k-\ell$ divides $n$, and that $C$ has $n/(k-\ell)$ edges; we call this the \emph{length} of $C$. 

Let $P$ be an $\ell$-path with vertices $v_1, \dots, v_n$, numbered as they appear in $P$. The \emph{ends} of $P$ are the ordered $\ell$-tuples $(v_1, v_2, \dots, v_\ell)$ and $(v_{n-\ell+1}, v_{n-\ell+2}, \dots, v_n)$. Our arguments frequently make use of the fact that if $P$ is an $\ell$-path with ends $a$ and $b$, and $Q$ is an $\ell$-path with ends~$b$ and~$c$ which has no vertices in common with $P$ outside $b$, then there is an $\ell$-path $PQ$ with vertex set $V(P) \cup V(Q)$ and edge set $E(P) \cup E(Q)$ which can be formed by concatenating the ordered sequences of vertices of $P$ and $Q$ (only including the vertices of $b$ once). In a similar manner we can merge larger collections of $\ell$-paths which share ends appropriately into a single long $\ell$-path or $\ell$-cycle.

We now establish further key properties of $k$-graph $\ell$-cycles and $\ell$-paths.

\begin{proposition} \label{claimconstruct}
    If $C$ is a $k$-uniform $\ell$-cycle or $\ell$-path on $n$ vertices, then there exists a partition $\mathcal{P}$ of $V(C)$ into sets each of size $\ctmod{k}{\ell}$, and possibly one smaller set, such that for each $S \in \mathcal{P}$ there is an edge $e \in E(C)$ with $S \subseteq e$.
\end{proposition}

\begin{proof}
    Enumerate the vertices of $C$ as $v_1,\dots, v_n$, in order as they appear in $C$, in such a way that $\{v_1,\dots, v_k\}$ is an edge of $C$.
    For each $1\leq i \leq \lceil{\frac{n}{\ctmod{k}{\ell}}}\rceil$ set
    $$V_i := \left\{v_{(i-1)\ctmod{k}{\ell}+1}, \dots, v_{i\ctmod{k}{\ell}}\right\},$$
    ignoring vertices which do not exist (because they are $v_j$ for some $j > n$). 
    So $|V_i| = \ctmod{k}{\ell}$ for each $i < \frac{n}{\ctmod{k}{\ell}}$, and $|V_i| \leq \ctmod{k}{\ell}$ for $i = \lceil\frac{n}{\ctmod{k}{\ell}}\rceil$. 
    Observe that the sets~$V_i$ partition $V(C)$. So it remains only to show that for each $V_i$ there is an edge $e \in E(C)$ with $V_i \subseteq e$, which holds since for each $1 \leq i \leq \lceil\frac{n}{\ctmod{k}{\ell}}\rceil$ the vertex $v_{(i-1)\ctmod{k}{\ell}+1}$ is the first vertex of an edge of $C$. 
\end{proof}

Recall that a \emph{strong independent set} in a $k$-graph $H$ is a set $S \subseteq V(H)$ such that every edge of $H$ has at most one vertex in $S$. This is a dual notion to a \emph{vertex cover} of $H$, which is a set $S \subseteq V(H)$ for which every edge of $H$ has at least one vertex in $S$.

\begin{proposition} \label{indepsetsinpath}
    Let $P$ be a $k$-uniform $\ell$-path with $n$ vertices. The maximum size of a strong independent set in $P$ is $\ceil{n/\floor{\frac{k}{k-\ell}}(k-\ell)}$. Moreover $P$ has a strong independent set $X$ of this size which is also a vertex cover of $P$. 
\end{proposition}

\begin{proof}
By Proposition~1.4 the vertices of $P$ may be partitioned into sets of size $\floor{\frac{k}{k-\ell}}(k-\ell)$ and possibly one smaller set in such a way that each set is a subset of an edge of $P$. The number of sets in this partition is then $\ceil{n/\floor{\frac{k}{k-\ell}}(k-\ell)\rfloor}$. Since every strong independent set in $P$ has at most one vertex from each set, this gives the desired bound.

Now let $v_1, \dots, v_n$ be the vertices of $P$, ordered as they appear in $P$. Define
$$X := \left\{v_{k-\ell+ (i-1) \floor{\frac{k}{k-\ell}}(k-\ell)}: 1 \leq i \leq \left\lceil{\frac{n}{\floor{\frac{k}{k-\ell}}(k-\ell)}}\right\rceil\right\}.$$
The set $X$ then has size $\ceil{n/\floor{\frac{k}{k-\ell}}(k-\ell)}$. Moreover $X$ is a strong independent set in $P$, and every edge of $P$ contains a vertex of $X$. 
\end{proof}

It is straightforward to see that every $\ell$-path $P$ is $k$-partite: if $v_1, \dots, v_n$ are the vertices of $P$ in order as they appear in $P$, then taking $V_i := \{v_j : j \in [n], j \equiv i \pmod{k}\}$ gives valid vertex classes for $P$. However, these vertex classes all have approximately the same size. Our final preliminary result describes how we may choose vertex classes for $P$ in a more unbalanced manner. Specifically, we want one class to be essentially as large as possible.  

\begin{proposition} \label{unbalancedpath}
Let $P$ be a $k$-uniform $\ell$-path with $n$ vertices. Let $w_1 = k-1$ and $w_2=\dots=w_k=\ctmod{k}{\ell} - 1$, and set $W = \sum_{i \in [k]} w_i$. The vertex set $V(P)$ may be partitioned into sets $U_1, \dots, U_k$ so that 
\begin{enumerate}[label=(\alph*), noitemsep]
    \item for each $i \in [k]$ every edge of $P$ has precisely one vertex in $U_i$, and 
    \item for each $i \in [k]$ we have $|U_i| = \frac{w_i}{W}n \pm 2$.
\end{enumerate}
\end{proposition}

\begin{proof}
Observe that $W = (k-1) + (k-1)\left(\ctmod{k}{\ell} - 1\right) = (k-1)\ctmod{k}{\ell}$. 
Let $U_1$ be a strong independent set in $P$ of size $\ceil{n/\floor{\frac{k}{k-\ell}}(k-\ell)}$ which is also a vertex cover of $P$, which exists by Proposition~\ref{indepsetsinpath}. Observe that we then have $|U_1| = \frac{w_1}{W}n \pm 1$. Let $u_1, u_2, \dots u_{n-|U_1|}$ be the remaining vertices of $P$, numbered in the order that they appear in $P$. For each $2 \leq i \leq k$ set $U_i := \{u_j : j \in [n-|U_1|], j \equiv i \pmod{k-1}\}$. Then the sets $U_1, \dots, U_k$ partition $V(P)$. Moreover, since each edge of $P$ consists of $k$ consecutive vertices of $P$, precisely one of which is in $U_1$, each edge of $P$ has precisely one vertex in each set $U_i$. Finally, for each $2 \leq i \leq k$ we have
$$ |U_i| = \frac{n-|U_1|}{k-1} \pm 1 = \frac{1}{k-1}\left(n -\frac{w_1}{W}n \pm 1\right) \pm 1 = \frac{w_i}{W} n \pm 2,$$
where the first equality uses our earlier bound on $|U_1|$ and the second holds since $w_i = \frac{W-w_1}{(k-1)}$ by definition of the $w_i$ and $W$.
\end{proof}

We will use the following consequence of Proposition~\ref{unbalancedpath}: for every real number $x \geq 0$ there is an $\ell$-path $k$-graph $P$ with vertex classes $U_1, \dots, U_k$ such that $|U_i| = w_ix \pm k$ for each $i \in [k]$. Indeed, if $Wx < k$ then we may take $P$ to be a single edge, so the sets $U_i$ each contain one vertex. On the other hand, if $Wx \geq k$ then by choosing $n$ to be the smallest integer greater than or equal to $Wx$ with $n \equiv \ell \pmod{k-\ell}$ we have $n = Wx \pm (k-\ell)$, and then for each $i \in [k]$ the path $P$ obtained from Proposition~\ref{unbalancedpath} has 
$$|U_i| = \frac{w_i n}{W} \pm 2 = w_i x \pm \left(2 + \frac{(k-\ell)w_i}{W}\right) = w_i x \pm 3,$$ 
as desired.

\subsection{Optimality of Theorem~\ref{main}} \label{sec:optimality}
Theorem~\ref{main} is best possible up to the $\alpha n$ error term, as demonstrated by the following modification of a construction given by Halfpap and Magnan~\cite{HM}, which is a hypergraph analogue of the extremal example for Dirac's theorem in $2$-graphs. 

\begin{constr}\label{construction}
    Fix $k\geq 3$ and $1\leq \ell \leq k-1$. Let $A$ and $B$ be vertex-disjoint sets with $|A| = \frac{n}{\ctmod{k}{\ell}} + 1$ and $|B| = n - |A|$, so $|A \cup B| = n$. Let $H_\mathrm{ext}$ be the $k$-graph $H$ with vertex set $A \cup B$ whose edges are all sets $e \in \binom{V(H)}{k}$ with $|e \cap A| \in \{0,1\}$.
\end{constr}

Observe that $H_\mathrm{ext}$ has minimum positive codegree 
$$\delta^+(H_\mathrm{ext}) = |B| - (k-2) = \left(\dcover\right)n - (k-1)$$ 
and has no isolated vertices. Suppose that $C$ is a Hamilton $\ell$-cycle in $H_\mathrm{ext}$. By Proposition~\ref{claimconstruct} there is a partition $\mathcal{P}$ of $V(C) = V(H)$ with at most $\lceil \frac{n}{ \ctmod{k}{\ell}}\rceil$ parts so that for each $S \in \mathcal{P}$ there is an edge $e$ of $C$ with $S \subseteq e$. Since $|A| > \lceil \frac{n}{ \ctmod{k}{\ell}}\rceil$, by averaging there must be some $S \in \mathcal{P}$ with $|S \cap A| \geq 2$. This gives a contradiction, since $S$ is contained in some edge of $C$, but every edge of $H$ has at most one vertex of $A$ and therefore cannot contain $S$. So $H_\mathrm{ext}$ does not contain a Hamilton $\ell$-cycle.

\section{Proof of Theorem~\ref{main}}

Our proof of Theorem~\ref{main} proceeds by the \emph{absorbing} method, which has seen widespread use over recent years for problems involving the embedding of spanning structures in graphs and hypergraphs. In our setting this has three key components.

The first component is a connecting lemma, which states that for any pair of ordered $\ell$-tuples $x$ and $y$ of vertices of $H$ there is a short $\ell$-path $P$ in $H$ with ends $x$ and $y$. Moreover, we can require that the interior vertices of $P$ (that is, vertices not in the ends $x$ and~$y$) avoid a small set of `forbidden' vertices.

\begin{lemma}\label{connectinglemma}
Fix $k \geq 3$, $1\leq\ell<k$ and constants with $1/n_0 \ll \alpha \ll 1/k$. Let $H$ be a $k$-graph on $n \geq n_0$ vertices with $\delta^+(H) \geq \left(\dconnect \right) n + \alpha n$, and let $X \subseteq V(H)$ be a set with $|X| \leq \alpha n$. For each pair of disjoint ordered $\ell$-tuples $x = (x_1, \dots, x_\ell)$ and $y = (y_1, \dots, y_\ell)$ of vertices in $V(H)\setminus X$ with $\deg(x), \deg(y)>0$, there exists an $\ell$-path $P$ in $H$ of length $\ceil{\frac{k}{k-\ell}}$ with ends $x$ and $y$ such that $V(P)  \cap X \subseteq x \cup y$.
\end{lemma}

To prove Lemma~\ref{connectinglemma}, we extend the ends $x$ and $y$ to the path $P$ by choosing the remaining vertices one by one. Since each vertex in an $\ell$-path lies in the intersection of at most $\ceil{\frac{k}{k-\ell}}$ edges of the path, our minimum positive codegree condition suffices to do this provided we make sure we never need to extend a set with degree zero; each of our individual vertex choices are made with this in mind.

\begin{proof}
Let $Q$ be a $k$-uniform $\ell$-path with length $\ceil{\frac{k}{k-\ell}}$ and with ends $a = (a_1, \dots, a_\ell)$ and $b = (b_1, \dots, b_\ell)$. Observe that
\begin{enumerate}[noitemsep, label=(\roman*)]
    \item no edge of $Q$ intersects both $a$ and $b$, and
    \item every vertex of $Q$ is contained in at most $\ceil{\frac{k}{k-\ell}}$ edges of $Q$.
\end{enumerate}

For a partial function $\phi : V(Q) \to V(H)$ we write 
\begin{equation*}\phi(S) := \{\phi(x): x \in S \mbox{ and $\phi(x)$ is defined}\}
\end{equation*}
for the image of a set $S \subseteq V(Q)$ under $\phi$. In other words, vertices in $S$ with no defined image do not contribute to the image of $S$ under $\phi$. 

Let $\Phi$ be the set of all injective partial functions $\phi: V(Q) \to V(H)$ such that
\begin{enumerate}[noitemsep, label=(\alph*)]
    \item for each $i \in [\ell]$ we have $\phi(a_i) = x_i$ and $\phi(b_i) = y_i$, 
    \item for each $u \in V(Q) \setminus (a \cup b)$ for which $\phi(u)$ is defined we have $\phi(u) \notin X$, and
    \item for every $e \in E(Q)$ we have $\deg_H(\phi(e)) \geq 1$.
\end{enumerate}
The assumptions of our lemma, together with~(i), imply that the partial function with $\phi(a_i) = x_i$ and $\phi(b_i) = y_i$ for each $i \in [\ell]$, and $\phi(u)$ undefined for every other $u \in V(Q)$, satisfies ~(a),~(b) and~(c). So the set $\Phi$ is non-empty. It follows that $\Phi$ contains an element which is maximal with respect to inclusion; fix the partial function $\phi$ to be such an element.

Suppose for a contradiction that $\phi(v)$ is undefined for some $v \in V(Q)$, and let $e_1, \dots, e_q$ be the edges of $Q$ which contain $v$. 
For each $i \in [q]$ we have $|\phi(e_i)| \leq k-1$ because $\phi(v)$ is undefined, and $\deg_H(\phi(e_i)) \geq 1$ by~(c). So by Proposition~\ref{choosevs} there are at most $n - \delta(H) < n/\ceil{\frac{k}{k-\ell}} - \alpha n$ vertices $z \in V(H)$ for which we do not have $\deg_H(e_i \cup \{z\}) \geq 1$. It follows that at least 
$$|V(H) \setminus X| - q \left(\frac{n}{\ceil{\frac{k}{k-\ell}}} - \alpha n\right) > n - |X| - \ceil{\frac{k}{k-\ell}} \cdot \frac{n}{\ceil{\frac{k}{k-\ell}}} + 2\alpha n \geq \alpha n$$
vertices $z \in V(H) \setminus X$ have $\deg_H(e_i \cup \{z\}) \geq 1$ for every $i \in [q]$. Indeed, for the first inequality note that if $q \geq 2$ the calculation is straightforward using (ii), whilst if $q=1$ we obtain the desired bound with plenty of room to spare since $\ceil{\frac{k}{k-\ell}} \geq 2$. Since $|V(Q)| < \alpha n$, it follows that we may fix a vertex $z \in V(H) \setminus (X \cup \phi(V(Q)))$ with $\deg_H(e_i \cup \{z\}) \geq 1$ for every $i \in [q]$. Let $\phi': V(Q) \to V(H)$ be the partial function with $\phi'(u) = u$ for every $u \in V(Q)$ for which $\phi(u)$ is defined and also $\phi'(v) = z$. We then have $\phi' \in \Phi$, contradicting the maximality of $\phi \in \Phi$. 

We conclude that in fact $\phi(v)$ is defined for every $v \in V(Q)$, so $\phi$ is an injective (non-partial) function $\phi: V(Q) \to V(H)$ satisfying~(a),~(b) and~(c). Since a set $S \in \binom{V(H)}{k}$ has $\deg_H(S) = 1$ if $S \in E(H)$ and $\deg_H(S) = 0$ otherwise,~(c) implies that $\phi$ maps $Q$ to a copy $P$ of $Q$ in $H$,~(a) implies that the ends of $P$ are $x$ and $y$ and~(b) implies that $V(P) \cap X \subseteq (x \cup y)$. So $P$ is the desired $\ell$-path in $H$.
\end{proof}

The second component of the proof of Theorem~\ref{main} is an absorbing lemma, which asserts the existence of a short $\ell$-path $P$ in $H$ which can absorb any small set $S$ of vertices outside $P$, meaning that there exists an $\ell$-path~$P'$ with the same ends as $P$ and with vertex set $S \cup V(P)$  (assuming a necessary divisibility condition on $|S|$).

\begin{lemma}\label{absorbinglemma}
Fix $k \geq 3$, $1\leq\ell\leq k-1$ and constants with $1/n_0 \ll \beta \ll \alpha, 1/k$. If $H$ is a $k$-graph on $n \geq n_0$ vertices with $\delta^+(H) \geq \left(\dconnect \right) n + \alpha n$ and no isolated vertices, then there exists a set $A \subseteq V(H)$ with $|A| \leq \beta n$ and distinct ordered $\ell$-sets $x,y$ with the property that for every set $S \subseteq V(H) \setminus A$ for which $|S| \leq \beta^2 n$ and $k-\ell$ divides $|S|$ there is an $\ell$-path in~$H$ with vertex set $A \cup S$ and with ends $x$ and $y$. 
\end{lemma}

We prove Lemma~\ref{absorbinglemma} by using a randomised selection to obtain short path sections with a weaker absorbing property, then connecting these sections to obtain the desired absorbing path~$P$ (with vertex set $A$).

\begin{proof}
Let $t :=  (2k-1)\cdot(k-\ell) + \ell$. We say an ordered $t$-tuple $(v_1,\dots, v_t)$ of vertices of $H$ \textit{absorbs} a set $T \in \binom{V(H)}{k-\ell}$ if 
\begin{enumerate}[label=(\roman*), noitemsep]
    \item $v_1, \dots, v_t$ in that order, form an $\ell$-path $Q$ in $H$, and
    \item there is an $\ell$-path in $H$ with vertex set $\{v_1,\dots, v_t\}\cup T$ and with the same ends $(v_1, \dots, v_\ell)$ and $(v_{t-\ell+1}, \dots, v_t)$ as $Q$.
\end{enumerate}

We start by giving a lower bound on how many absorbing sequences are there for a set of $k-\ell$ vertices. For each $T \in \binom{V(H)}{k-\ell}$, let $\mathcal{A}(T)$ denote the set of ordered $t$-tuples of vertices of $H$ which absorb $T$.

\begin{claim}\label{claimabs}
    For each $T \in \binom{V(H)}{k-\ell}$ we have $\mathcal{A}(T) \geq \alpha^t n^{t}$.
\end{claim}


Arbitrarily label the vertices of $T$ as $v_{2kj}$ for $j\in[k-\ell]$ (we consider these vertices to have been chosen when speaking of `previously-chosen vertices'). For each $2k<i\leq 3k-1$, choose $v_i \in V(H) \setminus T$ so that $\deg(\{v_{2k}, v_{2k+1}, \dots, v_{i-1}, v_i\}) \geq 1$. By Proposition~\ref{choosevs}, we have at least $\delta^+(H) \geq \alpha n$ options for each choice. Next choose $v_{3k} \in N(v_{2k+1}, \dots, v_{3k-1})\setminus T$; again there are at least $\delta^+(H) \geq \alpha n$ options. For $i<2k-1$ choose the vertices $v_i$ in decreasing order of index, taking $v_i$ to be a previously-unchosen vertex in the intersection of the sets $N(v_{i+1}, \dots, v_{i+k-1})$ and $N(v_{i+1}, \dots, v_{i+k}) \setminus \{v_{2k}\}$. (In the case that the second set involves a neighbourhood of $k$ elements, we ignore it and choose $v_i$ in the first neighbourhood only.) Each of these sets has size at least $\delta^+(H)$, so the number of options is at least $2\delta^+(H) - n - (3k + |T|) \geq \alpha n$.

For each $3k < i \leq t+(k-\ell)$, choose a vertex $v_i \in N(v_{i-k+1}, \dots, v_{i-1}) \cap S$ which has not previously been chosen, where
\begin{enumerate}[noitemsep, label=(\alph*)]
    \item if $2kj<i<(2j+1)k$ for some $j\in[k-\ell]$, then $S = N(\{v_{i-k}, \dots, v_{i-1}\}\setminus \{v_{2kj}\})$,
     \item if $(2j-1)k<i<2kj$ for some $j\in[k-\ell]$, then $S = \{u : \deg(u,v_{(2j-1)k + 1}, \dots, v_{i-1},v_{2kj})\geq 1\}$, and
     \item if $ i = (2j-1) k$ for some $j\in[k-\ell]$, then $S = V(H)$.
\end{enumerate}
Again this means choosing $v_i$ in the intersection of two sets each of size at least $\delta^+(H)$, so the number of options is at least $2\delta^+(H) - n - (t+k-\ell) \geq \alpha n$. 

Having made all these choices, let $A := \{v_1, \dots, v_{t+k-\ell}\} \setminus T$, so $A$ contains all the $v_i$ except $v_{2kj}$ for $j \in [k-\ell]$.
Observe that the choices of the vertices $v_i$ ensure the vertices of $A \cup T$, ordered by index, form a tight path in $H$, but also that the vertices of $A$, ordered by index, form a tight path in $H$. Since $|A|$ and $|A \cup T|$ are both divisible by $k-\ell$, deleting edges from the latter tight path yields the desired $\ell$-path $Q$ witnessing that the vertices of $A$, ordered by index, form an element of $\mathcal{A}(T)$. Since we had at least $\alpha n$ options for each of the $t$ vertices not in $T$ that we chose, in total we obtain at least $\alpha^t n^t$ ordered $t$-tuples in $\mathcal{A}(T)$. This completes the proof of Claim~\ref{claimabs}.
\medskip
        
Let $c:= 3\beta^2/\alpha^t$, and choose a set $F$ of ordered $t$-tuples of vertices of $H$ at random by including each ordered $t$-tuple in $F$ with probability $p = cn^{1-t}$, independently of the choices made for each other ordered $t$-tuple. The random variables $|F|$ and $|\mathcal{A}(T) \cap F|$ for each $T \in \binom{V(H)}{k-\ell}$ are then distributed binomially with expectations $\mathbb{E}[|F|] = p n^t = cn$ and $\mathbb{E}[|\mathcal{A}(T) \cap F|] = p |\mathcal{A}(T)| \geq p\alpha^t n^t = c\alpha^t n$. So by Theorem~\ref{chernoff} we have with high probability that $|F| \leq 2cn$ and that for every $T \in \binom{V(H)}{k-\ell}$ we have $|\mathcal{A}(T) \cap F| \geq c \alpha^t n/2$.

Let $I(F)$ be the number of intersecting pairs in $F$, that is, of ordered $t$-tuples $F_1, F_2 \in F$ such that $F_1\cap F_2 \neq \emptyset$. The random variable $|I(F)|$ then has expectation $\mathbb{E}[|I(F)|] \leq  p^2 \cdot t^2 \cdot n^t \cdot n^{t-1} = c^2t^2 n,$
so by Markov's inequality we have $|I(F)| \leq 2 c^2 t^2 n$ with probability at least $1/2$. So we may fix an outcome of our random selection of $F$ so that all of these events hold, so
\begin{enumerate}[noitemsep, label=(\alph*)]
    \item $|F| \leq 2cn$,
    \item for each $T \in \binom{V(H)}{k-\ell}$ we have $|\mathcal{A}(T) \cap F| \geq c\alpha^t n/2$, and
    \item there are at most $2c^2 t^2  n$ pairs of tuples in $F$ with non-empty intersection.
\end{enumerate} 
Form $F' \subseteq F$ by deleting all ordered $t$-tuples which are not in any set $\mathcal{A}(T)$ and also deleting one tuple of each intersecting pair. Let $m := |F'| \leq |F| \leq 2cn$. Every ordered $t$-tuple in $F'$ is in $\mathcal{A}(T)$ for some $T \in \binom{V(H)}{k-\ell}$, and therefore supports an $\ell$-path $P^i$ in $H$. For each $i \in [m]$ let $a_i$ and $b_i$ be ends of $P^i$; by iterative application of Lemma~\ref{connectinglemma} we obtain vertex-disjoint $\ell$-paths $Q_1, \dots, Q_{m-1}$, each of length $\lceil \frac{k}{k-\ell} \rceil$, so that $Q_i$ has ends $b_i$ and $a_{i+1}$, and so that the interior vertices of each $Q_i$ do not intersect $\bigcup_{i \in [m]} V(P_i)$. Collectively these $\ell$-paths then form an $\ell$-path $P = P_1Q_1P_2Q_2\dots P_{m-1}Q_{m-1}P_m$ in $H$. Let $x := a_1$ and $y := b_m$ be the ends of $P$, and let $A = V(P)$; observe that $|A| \leq tm + km \leq 4tcn \leq \beta n$. Now consider a set $S \subseteq V(H) \setminus A$ with $|S| \leq \beta^2 n$ such that $k-\ell$ divides $|S|$. Arbitrarily partition $S$ into $(k-\ell)$-tuples $S_1, \dots, S_r$. For each $i \in [r]$ we have 
$$|F' \cap \mathcal{A}(S_i)| \geq |F \cap \mathcal{A}(S_i)| - 2c^2t^2 n \geq c \alpha^t n/2 - 2c^2t^2 \geq \beta^2 n \geq r,$$
so we may greedily assign each $S_i$ to an ordered $t$-tuple in $F' \cap \mathcal{A}(S_i)$ in such a way that no two sets are assigned to the same ordered $t$-tuple. Finally, for each $i \in [r]$ replace the corresponding path section $P_j$ of $P$ by an $\ell$-path $P'_j$ in $H$ with vertex set $V(P_j) \cup S_i$ and with the same ends as~$P_j$ (this exists since the ordered sequence of vertices of $P_j$ is in $\mathcal{A}(S_i)$). This transforms $P$ into an $\ell$-path $P'$ in $H$ with vertex set $A \cup S$ and with ends $x$ and $y$. 
\end{proof}

The final component of the proof of Theorem~\ref{main} is a path tiling lemma, which asserts the existence of a constant number of pairwise vertex-disjoint $\ell$-paths in $H$ which collectively cover almost all of the vertices of $H$.

\begin{lemma}\label{coverlemma}
Fix $k \geq 3$, $1\leq\ell\leq k-1$ and constants with $1/n_0 \ll 1/M \ll \beta, \alpha \ll 1/k$. If $H$ is a $k$-graph on $n \geq n_0$ vertices with $\delta^+(H) \geq \left(\dcover\right)n + \alpha n$ and no isolated vertices, then there exists a collection~$\mathcal{P}$ of at most $M$ pairwise vertex-disjoint $\ell$-paths in $H$ such that $|\bigcup_{P\in\mathcal{P}} V(P)| \geq (1-\beta) n$.
\end{lemma}

The proof of Lemma~\ref{coverlemma} is more involved, so we defer it to Section~\ref{sec:tiling}. Observe that Lemma~\ref{coverlemma} uses the same minimum positive codegree condition as Theorem~\ref{main}; a simple modification of Construction~\ref{construction} demonstrates that the lemma would not hold if the expression $\dcover$ was replaced by any smaller constant.

We now combine the three components presented above to prove Theorem~\ref{main}. Loosely speaking, our approach is as follows. First we apply the absorbing lemma to obtain an absorbing $\ell$-path $P_0$ in $H$. We then set aside a small randomly-selected set $W$ of vertices of $H$, then apply the path tiling lemma to find $\ell$-paths $P_1, \dots, P_t$ in $H$ which are vertex-disjoint from each other and from $P_0$ and which do not use any vertex of $W$. Finally we use the connecting lemma to connect the paths $P_0, P_1, \dots, P_t$ into a single long $\ell$-cycle $C$ in $H$, with the connections being comprised of vertices from $W$. The cycle $C$ then includes all vertices of $H$ except for a small set $X$, and the absorbing property of $P_0$ allows us to replace $P_0$ in $C$ by an $\ell$-path $P_0'$ with the same ends as $P_0$ and with vertex set $V(P_0) \cup S$. This gives a Hamilton $\ell$-cycle $C'$ in $H$.

\begin{proof}[Proof of Theorem~\ref{main}]
Introduce new constants $c, \gamma$ and $\beta$ with $1/n_0 \ll 1/c \ll \gamma \ll \beta \ll \alpha, 1/k$. Write $\delta := \dcover$, and let $H$ be a $k$-graph on $n \geq n_0$ vertices with minimum positive codegree $\delta^+(H) \geq \delta n + \alpha n$ and no isolated vertices.

Apply Lemma~\ref{absorbinglemma} to obtain an absorbing set $A \subseteq V(H)$ with $|A| \leq \beta n$ and distinct ordered $\ell$-sets $x_0$ and $y_0$ with the property that for every $S \subseteq V(H)$ with $|S| \leq \beta^2 n$ such that $k-\ell$ divides $|S|$ there is an $\ell$-path in $H$ with vertex set $A \cup S$ and with ends $x$ and $y$. In particular, taking $S = \emptyset$ we find that there is an $\ell$-path $P_0$ in $H$ with ends $x_0$, $y_0$ and with vertex set $A$. 

Let $V' := V(H) \setminus A$, $H' := H[V']$ and $n' := |V'| = n - |A|$, so $H'$ is a $k$-graph on $n'$ vertices with $\delta^+(H') \geq \delta^+(H) - |A| \geq \delta n' + \alpha n'/2$. Moreover, by Proposition~\ref{choosevs}(b) there are no isolated vertices in $H'$. Select a random subset $W \subseteq V'$ by including each vertex of $V'$ in $W$ with probability $\gamma$, independently of all other choices, so that $|W|$ and $\deg(T, W)$ for each $T \in \binom{V(H)}{k-1}$ are binomial random variables. By Theorem~\ref{chernoff} and Proposition~\ref{choosevs} it follows with high probability that
\begin{enumerate}[noitemsep, label=(\roman*)]
    \item $(1 - \alpha/10) \gamma n' \leq |W| \leq (1 + \alpha/10) \gamma n'$,
    \item for each set $T \in \binom{V(H)}{k-1}$ we have $\deg_{H'}(T, W) \geq \gamma \deg_{H'}(T) - \gamma \alpha n'/4$, and
    \item every set $S \subseteq V(H)$ with $\deg_{H'}(S) \geq 1$ has $\deg_{H[S \cup W]} \geq 1$.
\end{enumerate}
Fix an outcome of $W$ for which these events occur, and set $V'' := V' \setminus W$, $H'' := H[V'']$ and $n'' = |V''| = n' - |W|$, so $H''$ is a $k$-graph on $n''$ vertices with $\delta^+(H'') \geq \delta^+(H') - |W| \geq \delta n'' + \alpha n''/3$. Apply Lemma~\ref{coverlemma} to choose a collection $P_1, \dots P_t$ of $t \leq c$ vertex-disjoint $\ell$-paths in $H''$ with $|\bigcup_{i=1}^t V(P_i)| \geq (1-\beta^2/2) n''$. For each $i \in [t]$ let $x_i$ and $y_i$ be the ends of $P_i$. 

We now iteratively choose $\ell$-paths $Q_0, Q_1, \dots, Q_t$ in $H$, each of length $\ceil{\frac{k}{k-\ell}}$, such that for each $0 \leq i \leq t$ we have that
\begin{enumerate}[noitemsep, label=(\alph*)]
    \item $Q_i$ has ends $y_i$ and $x_{i+1}$, where we take $x_{t+1}$ to mean $x_0$, 
    \item all interior vertices of $Q_i$ are contained in $W$ (so are not in $\bigcup_{i=1}^t V(P_i)$), and
    \item $Q_i$ does not contain any vertex of $Q_0, Q_1, \dots, Q_{i-1}$.
\end{enumerate}
To see that we can do this, suppose that we have chosen $\ell$-paths $Q_0, Q_1, \dots, Q_{i-1}$ satisfying (a), (b) and (c), and that we now wish to choose $Q_i$. Set $W_i = W \cup y_i \cup x_{i+1}$ and $X_i = \bigcup_{j=0}^{i-1} V(Q_j)$. So by (i) and (ii) we have
\begin{align*}
    \delta^+(H[W_i]) &\geq \gamma \delta^+(H') - \frac{\gamma \alpha n'}{4} \geq \gamma \left(\delta n' + \frac{\alpha n'}{2}\right) - \frac{\gamma \alpha n'}{4}
    \\ &\geq \left(\delta + \frac{\alpha}{10}\right) \left(1 + \frac{\alpha}{10}\right) \gamma n'\geq \left(\delta + \frac{\alpha}{10}\right) |W_i|,
\end{align*}
and 
$$|X_i| \leq t\cdot \left(\ceil{\frac{k}{k-\ell}}(k-\ell) + \ell\right) \leq 3ck \leq \frac{\alpha |W_i|}{10}.$$ 
Moreover, since $y_i$ is an end of $P_i$ and $x_{i+1}$ is an edge of $P_{i+1}$ we have $\deg_{H'}(y_i), \deg_{H'}(x_{i+1}) \geq 1$, and so by (iii) we have $\deg_{H[W_i]}(y_i), \deg_{H[W_i]}(x_{i+1}) \geq 1$. 
It follows by Lemma~\ref{connectinglemma} (applied with $H[W_i]$, $\alpha/10$ and $|W_i|$ in place of $H$, $\alpha$ and $n$ respectively) that there exists an $\ell$-path $Q_i$ in $H[W_i]$ of length $\ceil{\frac{k}{k-\ell}}$ with ends $y_i$ and $x_{i+1}$ and with no interior vertex in $X_i$, as required.

Observe that (a), (b) and (c) ensure that the $\ell$-paths $P_0Q_0P_1Q_1\dots P_tQ_t$, in that order, form an $\ell$-cycle~$C$ in $H$. Moreover each vertex of $H$ not covered by $C$ is either one of the at most $(\beta^2/2)n'' \leq \beta^2n/2$ vertices of $H''$ not covered by $P_1, \dots, P_t$, or one of the at most $2 \gamma n' \leq 2 \gamma n$ vertices of $W$. We conclude that $|V(H) \setminus V(C)| \leq \beta^2n/2 + 2 \gamma n \leq \beta^2 n$. Moreover, since $k-\ell$ divides both $n$ (by assumption) and $|C|$ (since $C$ is an $\ell$-path) we know that $k-\ell$ divides $|V(H) \setminus V(C)|$. So the absorbing property of $P_0$ implies that there is a path $P_0'$ with vertex set $V(P_0) \cup (V(H) \setminus V(C))$ and the same ends $x_0$ and $y_0$ as $P_0$. It follows that the cycle $C'$ formed by the paths $P_0'Q_0P_1Q_1\dots P_tQ_t$ is a Hamilton $\ell$-cycle in $H$.
\end{proof}

\section{Path tiling} \label{sec:tiling}

In this section we prove Lemma~\ref{coverlemma}, our path tiling lemma. This argument makes use of weak hypergraph regularity, for which we make the following definitions.  Let $H$ be a $k$-graph and let $V_1,\dots, V_k$ be pairwise-disjoint subsets of $V(H)$. The \textit{density} of the $k$-tuple $\{V_1, \dots, V_k\}$ is defined to be $$d(V_1,\dots, V_k) = \frac{e(V_1,\dots, V_k)}{\prod_{i=1}^k |V_i|},$$
where $e(V_1,\dots, V_k)$ denotes the number of edges in $H$ that have exactly one vertex in $V_i$ for each $i\in[k]$.
For $\varepsilon>0$ and $d\geq 0$, the $k$-tuple $\{V_1, \dots, V_k\}$ is said to be \emph{$(d, \varepsilon)$-regular} if $|d(V_1',\dots, V_k') - d| \leq \varepsilon$ holds for every $k$-tuple $\{V_1', \dots, V_k'\}$ that satisfies, for each $i\in[k]$, that $V_i'\subseteq V_i$ and $|V_i'|\geq \varepsilon |V_i|$. We say that $\{V_1, \dots, V_k\}$ is \emph{$\varepsilon$-regular} if there is some $d>0$ such that $\{V_1,\dots, V_k\}$ is $(d, \varepsilon)$-regular.

The following theorem is a $k$-graph version of the Szemer\'edi regularity lemma in graphs; the proof is essentially identical to the proof for the graph case.

\begin{theorem}[Weak hypergraph regularity lemma]\label{weakreg}
    For all $\varepsilon > 0$ and $k, t_0 \in \mathbb{N}$ there exist $T_0, n_0 \in\mathbb{N}$ such that for every $k$-graph $H$ on $n \geq n_0$ vertices there is an integer $t$ with $t_0 \leq t \leq T_0$ and a partition of $V(H)$ into parts $V_0, V_1, \dots, V_t$ such that 
    \begin{enumerate}[noitemsep, label=(\roman*)]
        \item $|V_1| = \dots = |V_t|$ and $|V_0|\leq \varepsilon n$, and
        \item all but at most $\varepsilon \binom{t}{k}$ of the $k$-tuples chosen from $\{V_1, \dots, V_t\}$ are $\varepsilon$-regular.
    \end{enumerate}
\end{theorem}

We call a partition of $V(H)$ into parts $V_0, V_1, \dots, V_t$ which satisfies (i) and (ii) an \textit{$\varepsilon$-regular partition}. We also refer to the sets $V_1, \dots, V_t$ the \emph{clusters} of $H$, and $V_0$ as the \emph{exceptional set} of $H$. Given an $\varepsilon$-regular partition $V(H)$ into parts $V_0, V_1, \dots, V_t$, and a parameter $d \geq 0$, we define the associated \textit{$(d, \varepsilon)$-cluster $k$-graph} of $H$ to be the $k$-graph $R = R(d, \varepsilon)$ with vertex set~$[t]$ whose edges are all sets $S = \{i_1, \dots, i_k\} \in \binom{[t]}{k}$ with the property that the $k$-tuple $\{V_{i_1},\dots, V_{i_k}\}$ is $\varepsilon$-regular with $d(V_{i_1}, \dots, V_{i_k}) \geq d$.

We will work with the cluster $k$-graph arising from an application of Theorem~\ref{weakreg} to $H$, using the following tool presented by Halfpap, Lemons and Palmer~\cite{HLP} (their statement was more precise about the relationship between the constants involved). This shows that if any small set of edges is deleted from a $k$-graph of large minimum positive codegree, then we can delete a further small set of edges to obtain a subgraph with almost as large minimum positive codegree. 

\begin{lemma}\label{subgraph}\cite[Lemma 18]{HLP}
    Suppose that $1/n \ll \varepsilon  \ll \gamma \ll 1/k$, and let $H$ be a $k$-graph on $n$ vertices. If $H_1$ is a subgraph of $H$ formed by the deletion of at most $\varepsilon n^k$ edges of $H$, then there exists a subgraph $H_2$ of $H_1$ with $\delta^+(H_2)\geq\delta^+(H)- \gamma n.$ Moreover, $H_2$ can be obtained from $H$ by the deletion of at most $\gamma n^k$ edges.
\end{lemma}

Our next lemma uses Lemma~\ref{subgraph} to show that if $H$ is a $k$-graph on $n$ vertices with no isolated vertices, and~$R$ is the cluster $k$-graph with $t$ vertices obtained following an application of Theorem~\ref{weakreg} to $H$, then~$R$ contains an almost-spanning subgraph $R'$ whose minimum positive codegree (as a proportion of the size of the vertex set) is almost as large as that of $H$.

\begin{lemma}\label{clusterdeg}
Suppose that $1/n_0 \ll 1/t \ll \varepsilon, d \ll \gamma \ll c, 1/k$ and that $q$ is an integer with $q \leq \gamma n/3$.
Let $H$ be a $k$-graph on $n\geq n_0$ vertices with $\delta^+(H) \geq cn$ and with no isolated vertices. Let $V_0, V_1, \dots, V_t$ be the sets of an $\varepsilon$-regular partition of $V(H)$, and let $R$ be the $(d, \varepsilon)$-cluster $k$-graph of $H$, so $V(R) = [t]$.  
We then have that $R$ contains a subgraph $R^*$ with $t'\geq (1-\gamma)t$ vertices, where $t'$ is divisible by $q$, such that $R^*$ has no isolated vertices and satisfies $\delta^+(R^*) \geq (c-\gamma)t'$.    
\end{lemma}

\begin{proof}
We begin by deleting edges of $H$ to obtain a subgraph $H_1$. Specifically, we delete the following edges.
    \begin{enumerate}[label=(\alph*), noitemsep]
        \item All edges intersecting $V_0$, of which there are at most $\varepsilon n^{k}$.
        \item All edges which intersect some $V_i$ more than once, of which there are at most $n^k/t$. Indeed, for each fixed $i \in [t]$, at most $|V_i| n^{k-2} \leq n^k/t^2$ edges intersect $V_i$ at least twice.
        \item All edges spanning a $k$-tuple $(V_{i_1}, \dots, V_{i_k})$ that is not $\varepsilon$-regular, of which there are at most $\varepsilon n^k$. Indeed, at most $\varepsilon\binom{t}{k}$ $k$-tuples of clusters are not $\varepsilon$-regular, each spanning less than $(n/t)^k$ edges.
        \item All edges spanning a $k$-tuple $(V_{i_1}, \dots, V_{i_k})$ with $d(V_{i_1}, \dots, V_{i_k})<d$, of which there are at most $dn^k$. Indeed, each such $k$-tuple supports at most $d (n/t)^k$ edges, and there are at most $t^k$ $k$-tuples of clusters. 
    \end{enumerate}

    Let $\varepsilon' = 2\varepsilon + d + 1/t$, so in total we delete at most $\varepsilon' n^k$ edges from $H$. Let $\gamma' := \gamma/(4k!c^{1-k})$ (so we may assume $\varepsilon \ll \gamma'$). By Lemma \ref{subgraph}, there is a subgraph $H_2$ of $H_1$, which can be obtained from $H$ by the deletion of at most $\gamma' n^k$ edges of $H$, with $\delta^+(H_2) \geq (c-\gamma')n$.

    Let $R'$ be the $k$-graph with vertex set $[t]$ whose edges are all sets $(i_1, \dots, i_k) \in \binom{[t]}{k}$ for which $H_2$ contains an edge spanning $(V_{i_1}, \dots, V_{i_k})$. Observe that the deletion of edges specified in (c) and (d) ensures that $R'$ is a subgraph of $R$.
    Now let $S' \in \binom{[t]}{k}$ have $\deg_{R'}(S') \geq 1$. This means there exists $j \in [t]$ such that $S' \cup \{j\}$ is an edge of $R'$, meaning that we may choose a vertex $v_i \in V_i$ for each $i \in [S']$ and a vertex $u \in V_j$ such that $\{v_i :i \in S'\} \cup \{j\}$ is an edge of $H_2$. Taking $S = \{v_i :i \in S'\}$, we then have $\deg_{H_2}(S) \geq 1$, so $\deg_{H_2}(S) \geq \delta^+(H_2) \geq (c-\gamma')n$. Furthermore, the deletion of edges in (a) and (b) ensures that no neighbours of $S$ in $H_2$ are in $V_0$ or in $V_i$ for some $i\in S'$. Since each cluster has size at most $n/t$, we deduce that $S$ has neighbours in at least $(c-\gamma')t$ clusters $V_i$ with $i\notin S\cup\{0\}$. So $S'$ has $\deg_{R'}(S') \geq (c-\gamma')t$, and we conclude that $\delta^+(R') \geq (c-\gamma')t$.

    We now show that $R'$ has at most $\gamma t/2$ isolated vertices. Indeed, each vertex of $H$ is contained in at least $\binom{cn}{k-1}$ edges of $H$ by Proposition~\ref{choosevs}. Since $H_2$ can be obtained from $H$ by deleting at most $\gamma' n^k$ edges from $H$, we deduce that at most $k\gamma' n^k/\binom{cn}{k-1} \geq k!c^{1-k}\gamma' n$ vertices of $H_2$ are isolated. Since $i \in [t]$ is an isolated vertex of $R'$ only if all of the at least $(1-\varepsilon) n/t \geq n/2t$ vertices of $V_i$ are isolated in $H_2$, it follows that at most $2k!c^{1-k}\gamma' t = \gamma t/2$ vertices of $R'$ are isolated. 
    
    Finally, let $R^*$ be the subgraph of $R$ formed by deleting all isolated vertices of $R'$ as well as at most $q$ arbitrary additional vertices, the number of which is chosen so that $q$ divides $t' := |V(R^*)|$. So $t' \geq t - \gamma t/2 - q \geq (1-\gamma) t$. For each vertex $j \in R^*$, since $j$ was not isolated in $R'$ it follows from Proposition~\ref{choosevs} that $j$ is not isolated in $R^*$, so $R^*$ has no isolated vertices. Likewise, since we deleted at most $\gamma t/2 + q \leq 5\gamma t/6$ vertices of $R'$ we have 
    $\delta^+(R^*) \geq \delta^+(R'') - 5 \gamma t/6 \geq (c - \gamma' - 5\gamma/6) t' \geq (c-\gamma)t$, as required.
\end{proof}

Our next lemma shows that a $k$-tuple of clusters with high density must contain a long $\ell$-path. Moreover, it allows us a degree of control over which vertices appear in which clusters. 

\begin{lemma}\label{longpath}
    Fix $k\geq 3$, $\ell \in [k-1]$ and $d \in (0,1)$. Let $H$ be a $k$-partite $k$-graph with vertex classes $X_1, \dots, X_k$ each of size $m$. Let $P$ be a $k$-uniform $\ell$-path and let $\{U_1, \dots, U_k\}$ be a partition of $V(P)$ with respect to which $P$ is $k$-partite. If $d(X_1, \dots, X_k) \geq d$ and $|U_i| \leq dm/k$ for each $i \in [k]$, 
    then $H$ contains a partition-respecting copy of $P$.
\end{lemma}

 \begin{proof} 
We first use an iterative argument to show that $H$ has a non-empty subgraph $H'$ with minimum positive codegree $\delta^+(H') > dm/k$. For this, we track a collection $\mathcal{S} \subseteq \binom{V(H)}{k-1}$ which initially consists of all sets $S \in \binom{V(H)}{k-1}$ with $\deg(S) \geq 1$. Note that since $H$ is $k$-partite every set in $S$ has $|S \cap V_i| \leq 1$ for each $i \in [k]$. In particular, we initially have $|\mathcal{S}| \leq km^{k-1}$. Our iterative argument will repeatedly remove elements of $\mathcal{S}$ and also edges of $H$. During this process we will keep track of the size of the number of \emph{incidences}, which are pairs $(e, S)$ where $e \in E(H)$, $S \in \mathcal{S}$ and $S \subseteq e$. Since initially we have $e(H) \geq dm^k$, and each edge is in precisely $k$ incidences, we initially have $|I| \geq dkm^k$, and so $|I|/|\mathcal{S}| \geq dm$.

We now repeat the following step: if there exists $S \in \mathcal{S}$ with $\deg(S) \leq dm/k$, then delete $S$ from $\mathcal{S}$ and delete every edge which contains $S$ as a subset from $H$. Observe that in this step the size of $\mathcal{S}$ is reduced by one, while the number of edges of $H$ is reduced by at most $dm/k$, meaning that the number of incidences is reduced by at most $dm$. So the ratio $|I|/|\mathcal{S}|$ does not decrease. Since $H$ has a finite number of edges, this process must terminate with a non-empty subgraph $H' \subseteq H$, at which point every $S \in \mathcal{S}$ has $\deg_{H'}(S) > dm/k$. Every $S \in \binom{V(H)}{k} \setminus \mathcal{S}$ was either not initially in $\mathcal{S}$, meaning that $\deg_H(S) = 0$ so $\deg_{H'}(S) = 0$, or was deleted from $\mathcal{S}$ at some point along with all edges containing it, so again we have $\deg_{H'}(S) = 0$. We conclude that $\delta^+(H') > dm/k$, as claimed.

To complete the proof, observe that the desired copy of $P$ may be chosen greedily: first choose an arbitrary edge of $H'$ for the first edge of $P$, then choose the remaining vertices one by one in the order that they appear in the path $P$. The point is that when we come to embed a vertex $x$ of $P$ in a cluster $V_i$, Proposition~\ref{choosevs} ensures that there are at least $\delta^+(H') > dm/k$ vertices $u \in V_i$ which are suitable for the embedding of $x$, and at most $|U_i| \leq dm/k$ vertices of $V_i$ have already been used for previously-considered vertices of $P$, so there is at least one vertex available for the embedding of $x$.
\end{proof}

At this point, the argument we would like to make to prove Lemma~\ref{coverlemma} is as follows. Fix $k \geq 3$ and $\ell \in [k-1]$, and let $H$ be a $k$-graph on $n$ vertices with no isolated vertices and $\delta^+(H) \geq \left(\dcover\right) n + \alpha n$. Apply Theorem~\ref{weakreg} to $H$ to obtain an $\varepsilon$-regular partition of $V(H)$ into parts $V_0, V_1, \dots, V_t$, and let $R$ be the associated $(d, \varepsilon)$-cluster $k$-graph of $H$ for a suitable parameter $d$. By Lemma~\ref{clusterdeg} we may choose an almost-spanning subgraph $R'$ of $R$ with no isolated vertices and with $\delta^+(H) \geq \left(\dcover\right) n + \alpha n/2$. Choose a perfect matching $M$ in $R'$, and for each edge $e \in M$ apply Lemma~\ref{longpath} repeatedly to choose vertex-disjoint $\ell$-paths in $H[\{V_i: i \in e\}]$ which cover approximately $(d-\varepsilon)m/k$ vertices in each cluster $V_i$ with $i \in e$. The point is that because $\{V_i: i \in e\}$ is an $\varepsilon$-regular $k$-tuple of clusters with density at least $d$, we know that for any subsets $U_i \subseteq V_i$ with $|U_i| \geq \varepsilon m$ for each $i \in e$ we have $d(V_i:i \in e) \geq d-\varepsilon$ (here $m$ denotes the common size of the clusters $V_1, \dots, V_t$). So we may keep choosing $\ell$-paths in $H[\{V_i: i \in e\}]$ until for some $i\in e$ all but at most $\varepsilon m$ vertices of $V_i$ are covered; since each path covers approximately the same number of vertices in each cluster $V_i$ with $i \in e$, it follows that almost all vertices of each $V_i$ with $i \in e$ are covered. Doing this for every $e \in M$ therefore gives a collection of $\ell$-paths which covers all vertices of $H$ except for the small number of vertices in clusters $V_j$ with $j \notin R^*$, the small number of vertices in $V_0$ and the small number of uncovered vertices in each cluster $V_j$ with $j \in R^*$, completing the proof.

The problem is that to ensure a perfect matching in a $k$-graph $H$ on $n$ vertices we need a minimum positive codegree condition of around $\delta^+(H) = \frac{k-1}{k}n$, as demonstrated by Halfpap and Magnan~\cite{HM}. In the case when $k-\ell$ divides $k$, this is the condition assumed by Lemma~\ref{coverlemma}, and so the above argument can be developed to a full proof for this case. Unfortunately, in the case when $k-\ell$ does not divide $k$ the minimum positive codegree condition assumed by Lemma~\ref{coverlemma} is weaker than the condition required for a perfect matching, and so the argument fails at this step. By a slight modification to the argument it would actually be sufficient to obtain only a matching in $R'$ covering almost all vertices of $R'$, or a perfect fractional matching in $R'$, but appropriate modifications of the extremal construction given by Halfpap and Magnan~\cite{HM} demonstrate that a weaker minimum positive codegree condition is not sufficient to ensure either of these.

Instead, we proceed by finding a perfect \emph{weighted} fractional matching in $R'$. A \textit{fractional matching} in a $k$-graph $H$ is a function $q:E(H) \to \mathbb{R}_{\geq 0}$ such that for every $v \in V(H)$ the sum of $q(e)$ over all edges $e$ which contain $v$ is at most~$1$; we say that $q$ is \emph{perfect} if this sum is precisely~$1$ for every vertex of $H$. A good way to think about this is that the edge $e$ places its weight $q(e)$ on each of its $k$ vertices and we require that the total weight placed on each vertex is at most~$1$ (or precisely~$1$ if $q$ is perfect). Now let $w_1, w_2, \dots, w_k \in \mathbb{R}_{\geq 0}$. A $(w_1, w_2, \dots, w_k)$-fractional matching in $H$ is conceptually similar, but now each edge $e$ places its weight upon the vertices it contains in an unbalanced manner. Specifically, for each $i \in [k]$ the edge $e$ places weight $w_iq(e)$ on the $i$th vertex of $e$; again we require that the total weight placed at each vertex is at most~$1$ (or precisely~$1$ if $q$ is perfect). This of course requires that we have an order on the vertices of $e$. For this reason our definition of a weighted fractional matching actually assigns a weight $q(e)$ to every permutation of every edge of $H$, rather than to edges as unordered sets.

This leads us to the following formal definitions. Let $H$ be a $k$-graph, and let $E^*(H)$ denote the set of \emph{ordered} edges of $H$, that is, the ordered $k$-tuples $(v_1, \dots, v_k)$ with $\{v_1, v_2, \dots, v_k\} \in E(H)$. So for each edge of $H$ we have $k!$ corresponding elements of $E^*(H)$. For each vertex $v \in V(H)$ and each $i \in [k]$ we write $E_i^v(H) \subseteq E^*(H)$ for the set of all ordered edges in $E^*(H)$ in which $v$ is the $i$th  element. Fix $w = (w_1,\dots,w_k) \in \mathbb{R}_{\geq 0}^k$. A \emph{$w$-weighted fractional matching} in $H$ is a function $q: E^*(H) \to \mathbb{R}$ such that $q(e) \geq 0$ for every $e \in E^*(H)$ and, for every $v \in V(H)$, we have 
$$\sum_{i \in [k]} \sum_{e \in E_i^v(H)} w_i \cdot q(e) \leq 1.$$ 
We say that $q$ is \emph{perfect} if we have equality in the latter inequality for every $v \in V(H)$, or equivalently, if $\sum_{e \in E^*(H)} q(e) = n/\sum_{i \in [k]} w_i $. Note that a (perfect) fractional matching in $H$ is identical to a (perfect) $(1, 1, \dots, 1)$-weighted fractional matching in $H$.


Our proof of Lemma~\ref{coverlemma} will use perfect $(w_1, w_2, \dots, w_k)$-weighted fractional matchings where $w_1 = k-1$ and $w_2 = w_3 = \dots = w_k = \ctmod{k}{\ell} - 1$. These weights are chosen to satisfy two purposes. Firstly, a perfect $(w_1, w_2, \dots, w_k)$-weighted fractional matching in the cluster $k$-graph $R$ of $H$ enables us to partition the clusters of $H$ into groups of subclusters of varying sizes for which we can find vertex-disjoint $\ell$-paths covering almost all vertices in the group. Secondly, the minimum positive codegree condition of Lemma~\ref{coverlemma} allows us to obtain a $(w_1, w_2, \dots, w_k)$-weighted fractional matching by an application of Farkas' lemma on solvability of finite systems of linear inequalities, as follows.

\begin{lemma}[Farkas' lemma]\label{farkas}
For every $\mathbf{v}\in\mathbb{R}^n$ and every finite set $\mathcal{X}\subseteq\mathbb{R}^n$, either there exist weights $\lambda_{\bf{x}} \geq 0$ for each $\bf{x} \in \mathcal{X}$ such that $\sum_{\bf{x} \in \mathcal{X}} \lambda_{\bf{x}} \bf{x} = \bf{v}$, or there exists $\mathbf{y}\in\mathbb{R}^n$ such that $\mathbf{y}\cdot\mathbf{x}\leq 0$ for every $\bf{x}\in\mathcal{X}$ and $\mathbf{y}\cdot\mathbf{v} > 0$.
\end{lemma}

\begin{lemma}\label{pfm}
Let $G$ be a $k$-graph on $n$ vertices, where $n$ is divisible by $\floor{\frac{k}{k-\ell}}(k-\ell)$, and set $w_1 := k-1$ and $w_i := \ctmod{k}{\ell} - 1$ for each $2 \leq i \leq k$. If $\delta^+(G) \geq \left(\dcover\right) n$ and $G$ has no isolated vertices, then $G$ admits a $(w_1, \dots, w_k)$-weighted perfect fractional matching.
\end{lemma}

\begin{proof}
     Let $v_1,\dots,v_n$ be the vertices of $G$ and set $\delta := \dcover$, so $\delta n$ is an integer. For each edge $e$ of $G$ and each $v_i\in e$, we define $\chi_{i}(e)\in\mathbb{R}^n$ to be the vector whose $i$th coordinate is $k-1$, whose $j$th coordinate is $\ctmod{k}{\ell}-1$ for each $j$ such that $v_j\in e\setminus\{v_i\}$, and with all other coordinates equal to $0$. Let $\mathcal{X}$ be the collection of vectors $\chi_{i}(e)$ for each $e\in E(G)$ and each $v_i\in e$.

     Suppose for a contradiction that $G$ does not admit a $(w_1, \dots, w_k)$-weighted perfect fractional matching. It follows that there are no weights $\lambda_{\bf{x}} \geq 0$ for each $\bf{x} \in \mathcal{X}$ with $\sum_{\bf{x} \in \mathcal{X}} \lambda_{\bf{x}} \bf{x} = \bf{1}$ (otherwise, using the $\lambda{\bf{x}}$ as edge weights would give us a $(w_1, \dots, w_k)$-weighted perfect fractional matching in $G$). So by Lemma~\ref{farkas} there exists $\mathbf{y} = (y_1, \dots, y_n) \in \mathbb{R}^n$ such that $\mathbf{y}\cdot \mathbf{1} > 0$ and $\mathbf{y}\cdot \mathbf{x} \leq 0$ for every $\mathbf{x}\in\mathcal{X}$. 
     
By reordering the vertices if necessary, assume without loss of generality that $y_i \geq y_j$ for all $i\geq j$. Since $G$ has no isolated vertices we have $\deg(v_n) \geq 1$, and so by repeated applications of Proposition~\ref{choosevs} we may choose $i_1, i_2, \dots, i_{k-1} \geq \delta n$ so that for each $1 \leq j \leq k-1$ we have $\deg(v_n, v_{i_1}, v_{i_2}, \dots, v_{i_j}) \geq 1$. Observe in particular that $e = \{v_n, v_{i_1}, \dots, v_{i_{k-1}}\}$ is an edge of $G$.

    Define $\mathbf{a} \in \mathbb{R}^n$ to be the vector whose $n$th coordinate is $k-1$, whose $j$th coordinate is $\ctmod{k}{\ell}-1$ for each $j\in\{i_1,\dots, i_{k-1}\}$, and all of whose other coordinates are $0$. Observe that we then have $\mathbf{a} = \chi_n(e) \in \mathcal{X}$, so $\mathbf{y}\cdot\mathbf{a} \leq 0$. It follows that
     \begin{align*}
         0 &\geq \frac{n}{(k-1)\ctmod{k}{\ell}}\mathbf{y}\cdot\mathbf{a} \\
         &= \frac{n}{(k-1)\ctmod{k}{\ell}} \left((k-1)y_n + \sum_{j=1}^{k-1}\left(\ctmod{k}{\ell}-1\right)y_{i_j})\right)\\
         &\geq (1-\delta)n\cdot y_n + \delta n \cdot y_{\delta n}
         \geq (y_n+y_{n-1}+\dots + y_{\delta n +1}) + (y_{\delta n} + y_{\delta n-1} + \dots + y_1)
         =\mathbf{y}\cdot\mathbf{1} > 0,
     \end{align*}
a contradiction. We conclude that $G$ admits a $(w_1, \dots, w_k)$-weighted perfect fractional matching, as required.
\end{proof}

We will use Lemma~\ref{pfm} through the following corollary, which states that with a slightly stronger minimum positive codegree condition we can additionally require that all edge weights in our weighted perfect fractional matching are small.

\begin{corollary}\label{pfm_low_weight}
Let $G$ be a $k$-graph on $n$ vertices, where $n$ is divisible by $\floor{\frac{k}{k-\ell}}(k-\ell)$, and set $w_1 := k-1$ and $w_i = \ctmod{k}{\ell} - 1$ for each $2 \leq i \leq k$. If $\delta^+(G) \geq \left(\dcover\right) n + \alpha n$ and $G$ has no isolated vertices, then $G$ admits a $(w_1, \dots, w_k)$-weighted perfect fractional matching~$q$ with $q(e) \leq 1/(\alpha n)$ for every $e \in E^*(G)$.
\end{corollary}

\begin{proof}
Let $Q$ be the set of all $(w_1, \dots, w_k)$-weighted perfect fractional matchings $q$ in $G$. For each $q \in Q$ let $m(q)$ be the maximum edge weight in $q$, that is, $m(q) := \max_{e \in E^*(G)} q(e)$. 
Let $M := \min_{q \in Q} m(q)$; we know that this minimum exists since the elements of $Q$ are solutions to a linear program with rational coefficients. So we may fix $q \in Q$ with $m(q) = M$. Since $G$ is finite it follows that there exists $\nu < M$ for which each edge $e \in E^*(G)$ has either $q(e) = M$ or $q(e) < M - \nu$. 

Let $G'$ be the subgraph of $G$ obtained by deleting all edges of $G'$ which have an ordering $e$ with $q(e) = M$. If $G'$ admits a $(w_1, \dots, w_k)$-weighted perfect fractional matching $q'$, then the weighted average $q''$ of $q$ and $q'$ defined by setting $q''(e) := (1-\nu) q(e) + \nu q'(e) < (1-\nu) (M-\nu) + \nu < M$ for every $e \in E^*(G')$ and $q''(e) := (1-\nu) q(e) < q(e) = M$ for every $e \in E^*(G) \setminus E^*(G')$ gives a $(w_1, \dots, w_k)$-weighted perfect fractional matching $q''$ in $G$ which contradicts our choice of $M$. (Here we used the fact that $q'(e) \leq 1$, which follows from the fact that $w_i \geq 1$ for each $i \in [k]$ and $q'$ is a $(w_1, \dots, w_k)$-weighted perfect fractional matching.)

We conclude that $G'$ has no $(w_1, \dots, w_k)$-weighted perfect fractional matching. It follows by Lemma~\ref{pfm} that either $\delta^+(G') < \left(\dcover\right) n$ or $G'$ has an isolated vertex. Either way there must exist a vertex $v$ which is contained in more than $\alpha n$ distinct edges of $G$ which are not edges of $G'$. Each such edge has an ordering $e$ with $q(e) = M$, so the total weight at $v$ is at least $M \cdot (\alpha n) \cdot \min_{i \in [k]} w_i \leq M \alpha n$. On the other hand, the total weight at $v$ is~$1$ by definition of a perfect weighted fractional matching, so $M < 1/(\alpha n)$ and therefore $q$ is as required.
\end{proof}

We are now ready to combine the tools we have developed to prove the main result of this section, our path tiling lemma. This proceeds in the same way as the approach sketched earlier in this section, except that we now use a weighted perfect fractional matching in $R'$ rather than a perfect matching.

\begin{proof}[Proof of Lemma~\ref{coverlemma}]
Introduce new constants satisfying 
$$\frac{1}{n_0} \ll \frac{1}{M} \ll \frac{1}{T_0} \ll \frac{1}{t_0} \ll \varepsilon \ll d \ll \gamma \ll \beta, \alpha \ll \frac{1}{k}.$$ 
Let $H$ be a $k$-graph on $n \geq n_0$ vertices with $\delta^+(H)\geq \left(\dcover \right) n + \alpha n$ and with no isolated vertices. 
Apply Theorem~\ref{weakreg} to obtain an integer $t$ with $t_0 \leq t \leq T_0$ and an $\varepsilon$-regular partition of $V(H)$ into parts $V_0, V_1, \dots, V_t$. In particular $|V_0| \leq \varepsilon n$; let $m := |V_1| = \dots = |V_t|$, so $m \geq (1-\varepsilon)n/t$.

    Let $R$ be the $(d, \varepsilon)$-cluster $k$-graph of $H$ with respect to this partition, so $R$ has vertex set~$[t]$. By Lemma~\ref{clusterdeg} there is a subgraph $R' \subseteq R$ on $t' \geq (1-\gamma)t$ vertices, where $t'$ is divisible by $\lfloor\frac{k}{k-\ell}\rfloor (k-\ell)$, with $\delta^+(R') \geq \left(\dcover\right) t' + \frac{\alpha t'}{2}$ and with no isolated vertices.

    Set $w_1 := k-1$ and $w_i := \ctmod{k}{\ell} - 1$ for each $2 \leq i\leq k$. By Lemma~\ref{pfm_low_weight} $R'$ admits a $(w_1, \dots, w_k)$-weighted perfect fractional matching $q$ in $R'$ with the property that for every ordered edge $e \in E^*(R')$ the weight $q(e) \geq 0$ assigned to $e$ satisfies $q(e) \leq 2/(\alpha t') \leq d/2k^2.$

    For each $e \in E^*(R')$ in turn we now do the following. Let $U_1, \dots, U_k$ be the clusters of $e$, in the order that they appear in $e$. Choose an $\ell$-path $P_e$ in $H[U_1, \dots, U_k]$ which, for each $i \in [k]$, covers $(1-\beta/2) q(e) w_i m \pm k$ vertices of $U_i$ . We also insist that $P_e$ should be vertex-disjoint from all $\ell$-paths previously selected for previously-considered $e \in E^*(R')$.
    
    To see that this is possible, for each $i \in [k]$ let $U'_i \subseteq U_i$ be a set of $\beta m/3$ vertices of~$U_i$ which have not been covered by previously-selected paths (we justify later why such sets must exist). Since $U'_1, \dots, U'_k$ is an edge of $R'$, and therefore is an edge of $R$, we know that $H[U_1, \dots, U_k]$ is $(d, \varepsilon)$-regular, and so $d(H[U'_1, \dots, U'_k]) \geq d-\varepsilon$. By Proposition~\ref{unbalancedpath} (and the consequence described in the text immediately following the proposition) there exists a $k$-uniform $\ell$-path $Q$ which is $k$-partite with respect to vertex classes $Y_1, \dots, Y_k$ whose sizes satisfy $|Y_i| = (1-\beta/2) q(e) w_i m \pm k$ for each $i \in [k]$. Since for each $i\in [k]$ we have $|Y_i| \leq  (1-\beta/2) q(e) w_i m + k \leq q(e)km \leq (d-\varepsilon) m/k$, by Lemma~\ref{longpath} there exists a partition-respecting copy of $Q$ in $H[U'_1, \dots, U'_k]$, which is our desired path $P_e$. 
        
    Having successfully chosen an $\ell$-path $P_e$ with this property for every $e \in E^*(R')$, for each $i \in V(R')$ the total number of vertices covered in the cluster $V_i$ is
    \begin{equation}\label{vscovered}
        \sum_{i \in [k]} \sum_{e \in E_i^U(R')} \left(\left(1-\frac{\beta}{2}\right)  q(e)w_i m \pm k \right) = \left(1-\frac{\beta}{2}\right) m \pm k^2(t')^k = m - \frac{\beta m}{2} \pm \frac{\beta m}{6},
    \end{equation}
    where the first equality follows from the definition of a perfect $w$-weighted fractional matching. The upper bound expressed in~\eqref{vscovered} ensures that at all times when choosing the $\ell$-paths $P_e$ there are at least $\beta m /3$ vertices in each cluster which have not yet been covered by previously-selected paths, justifying our earlier assumption that this was the case.
    Moreover, since $m \geq (1-\varepsilon) n/t$ and $t' \geq (1-\gamma) t$, the lower bound expressed in~\eqref{vscovered} implies that the selected paths $P_e$ collectively cover at least $t' (1-2\beta/3)m \geq (1-\varepsilon)(1-\gamma)(1-2\beta/3)n \geq (1-\beta) n$ vertices of $H$, as required. Finally, since we chose one $\ell$-path $P_e$ for each $e \in E^*(R')$, the number of paths $P_e$ is at most $t^k \leq M$.
\end{proof}

\bibliographystyle{plain}
\bibliography{biblio}
\end{document}